\documentclass[11pt]{amsart}

\usepackage{amsthm,amssymb,amstext,amscd,amsfonts,amsbsy,amsxtra,latexsym,amsmath,
amsrefs,
color,verbatim
}
\usepackage[english]{babel}
\usepackage[latin1]{inputenc}
\usepackage{mathrsfs}
\allowdisplaybreaks

\newcommand{\field}[1]{\mathbb{#1}}
\newcommand{\N}{\field{N}}
\newcommand{\Z}{\field{Z}}
\newcommand{\R}{\field{R}}
\newcommand{\C}{\field{C}}
\newcommand{\Q}{\field{Q}}

\renewcommand{\H}{\mathbb{H}}

\newcommand{\SL}{\operatorname{SL}}

\newcommand{\im}{\operatorname{Im}}
\newcommand{\re}{\operatorname{Re}}

\newcommand{\z}{\mathfrak{z}}

\numberwithin{equation}{section}
\newtheorem{theorem}{Theorem}
\numberwithin{theorem}{section}
\newtheorem{lemma}[theorem]{Lemma}
\newtheorem{prop}[theorem]{Proposition}
\newtheorem{corollary}[theorem]{Corollary}
\theoremstyle{remark}
\newtheorem*{remark}{Remark}

\theoremstyle{definition}
\newtheorem{definition}[theorem]{Definition}

\title{Niebur-Poincar\'e Series
 and Traces of Singular Moduli}
\author{Steffen L\"obrich}

\begin{document}

\maketitle

\begin{abstract}
We compute the Fourier coefficients of analogues of Kohnen and Zagier's modular forms $f_{k,\Delta}$ of weight $2$ and negative discriminant. These functions can also be written as twisted traces of certain weight $2$ Poincar\'e series with evaluations of Niebur-Poincar\'e series as Fourier coefficients. This allows us to study twisted traces of singular moduli in an integral weight setting. In particular, we recover explicit series expressions for twisted traces of singular moduli and extend algebraicity results by Bengoechea to the weight $2$ case. We also compute regularized inner products of these functions, which in the higher weight case have been related to evaluations of higher Green's functions at CM-points.
\end{abstract}
 
\section{Introduction}

For a positive discriminant $\Delta$ and an integer $k>1$, Zagier \cite{zamrq} introduced the weight $2k$ cusp forms (in a different normalization)
\begin{equation}\label{fkd}
f_{k,\Delta}(\tau):=\frac{\Delta^{k-\frac12}}{2\pi}\sum_{Q\in\mathscr{Q}_{\Delta}}Q(\tau, 1)^{-k},
\end{equation}
where $\mathscr{Q}_\Delta$ denotes the set of binary integral quadratic forms of discriminant $\Delta$. The functions $f_{k,\Delta}$ were extensively studied by Kohnen and Zagier and have several applications. For example, they used these functions to construct the kernel function for the Shimura and Shintani lifts and to prove the non-negativity of twisted central $L$-values \cite{kz2}. Furthermore, the even periods
$$
\int_0 ^\infty f_{k,\Delta}(it)t^{2n}dt,\quad (0\leq n\leq k-1)
$$
of the $f_{k,\Delta}$ are rational \cite{kz}. Bengoechea \cite{ben} introduced analogous functions for negative discriminants and showed that their Fourier coefficients are algebraic for small $k$. These functions are no longer holomorphic, but have poles at the CM-points of discriminant $\Delta$. They were realized as regularized theta lifts by Bringmann, Kane, and von Pippich \cite{bkcyc} and Zemel \cite{ze}. Moreover, Bringmann, Kane, and von Pippich related regularized inner products of the $f_{k,\Delta}$ to evaluations of higher Green's functions at CM-points. \\

The right-hand side of \eqref{fkd} does not converge for $k=1$. However, one can use Hecke's trick to obtain weight $2$ analogues of the $f_{k,\Delta}$. These were introduced by Zagier \cite{zamrq} and further studied by Kohnen \cite{ko}. The aim of this paper is to analyze these weight $2$ analogues for negative discriminants. Here we deal with generalizations $f^*_{d,D,N}$ for a level $N$, a discriminant $d$, and a fundamental discriminant $D$ of opposite sign (see Definition \ref{fdef}). The $f^*_{d,D,N}$ transform like modular forms of weight $2$ for $\Gamma_0(N)$ and have simple poles at the Heegner points of discriminant $dD$ and level $N$. Let $\mathscr{Q}_{dD, N}$ denote the set of quadratic forms $[a,b,c]$ of discriminant $dD$ with $a>0$ and $N|a$, $\chi_D$ the generalized genus character associated to $D$, and $H(d,D,N)$ the twisted Hurwitz class number of discriminants $d$, $D$ and level $N$ (see Subsection \ref{2.2} for precise definitions). Then we obtain the following Fourier expansion ($v:=\im(\tau)$ throughout).

\begin{theorem}\label{anafexp}
For $v>\frac{\sqrt{|dD|}}{2}$, we have 
$$
f^*_{d,D,N} (\tau)=-\frac{ 3 H(d,D,N)}{\pi \left[\SL_2(\Z):\Gamma_0(N)\right]v} - 2\sum_{n\geq 1}\sum_{a>0 \atop N|a }S_{d,D}(a,n)\sinh\left(\frac{\pi n\sqrt{|dD|}}{a}\right)e(n\tau),
$$
where $e(w):=e^{2\pi iw}$ for all $w\in\C$ and
$$
S_{d,D}(a,n):=\sum_{b\pmod{2a}\atop b^2\equiv dD\pmod{4a}}\chi_D\left(\left[a,b,\frac{b^2-dD}{4a}\right]\right)e\left(\frac{nb}{2a}\right).
$$
\end{theorem}

\begin{remark}
The exponential sums $S_{d,D}$ also occur for example in \cite{dit} and \cite{ko}.
\end{remark}

Note that we obtain a non-holomorphic term in the Fourier expansion of $f^*_{d,D,N}$, just like in the case of the non-holomorphic weight $2$ Eisenstein series $E_2 ^*$ (see Subsection \ref{nota}). Therefore, in contrast to the higher weight case, the $f^*_{d,D,N}$ are in general no longer meromorphic modular forms, but {\it polar harmonic Maass forms}. This class of functions is defined and studied in Subsection \ref{polar}. \\

We also use a different approach to compute the coefficients of the $f^*_{d,D,N}$, writing them as traces of certain Poincar\'e series denoted by $H_{N}^*(z,\cdot)$ (see Proposition \ref{anacontin}). The $H_{N}^*(z,\cdot)$ are weight $2$ analogues of Petersson's Poincar\'e series and were introduced by Bringmann and Kane \cite{bk} to obtain an explicit version of the Riemann-Roch Theorem in weight $0$. We obtain the following different Fourier expansion of the $f^*_{d,D,N}$, realizing their coefficients as twisted traces of the Niebur-Poincar\'e series $j_{N,n}$ (see Definition \ref{tracedef} and Theorem \ref{niebth}).

\begin{theorem}\label{fexp}
For $v>\operatorname{max}\left\{\frac{\sqrt{|dD|}}{2}, 1\right\}$, we have 
$$
f^*_{d,D,N}(\tau)=-\frac{ 3 H(d,D,N)}{\pi \left[\SL_2(\Z):\Gamma_0(N)\right]v}-\sum_{n>0}\operatorname{tr}_{d,D,N} \left(j_{N, n}\right)e(n\tau).
$$
\end{theorem}

An interesting phenomenon occurs when $\Gamma_0(N)$ has genus $0$. Subgroups of this type and their Hauptmoduln play a fundamental role in Monstrous Moonshine (see for example \cite{cn} for a classical and \cite{df} for a more modern treatment). When we apply the suitably normalized $n$-th Hecke operator $T_n$ to the Hauptmodul $J_N$ for $\Gamma_0(N)$, then the Niebur-Poincar\'e series $j_{N, n}$ coincides with $T_n J_N$, up to an additive constant. Zagier \cite{za} showed that, for discriminants $d<0$ and $D>0$, the functions
$$
q^{d}+\sum_{D>0}\operatorname{tr}_{d,D,N}(T_n J_N)q^D \quad \text{and}\quad q^{-D}+B_n(D,0)+ \sum_{d>0}\operatorname{tr}_{d,D,N}(T_n J_N)q^d 
$$
are weakly holomorphic modular forms for $\Gamma_0(4N)$ of weight $\frac12$ resp.~$\frac32$ in the Kohnen plus-space. Now summing over $n$ instead of $D$ or $d$, Theorem \ref{fexp} states that the twisted Hecke traces $\{\operatorname{tr}_{d,D,N}(T_n J_N)\}_{n>0}$ give rise to Fourier coefficients of the meromorphic modular forms $f^*_{d,D,N}-\frac{H(d,D,N)}{\left[\SL_2(\Z):\Gamma_0(N)\right]}E_2^*$. \\

We give three applications of Theorem \ref{fexp}. First, comparing Theorems \ref{anafexp} and \ref{fexp}, we obtain explicit series expressions for traces of Niebur-Poincar\'e series. 

\begin{corollary}\label{fserexp}
We have  
$$
\operatorname{tr}_{d,D,N} \left(j_{N, n} \right) = 2\sum_{a>0 \atop N|a }S_{d,D}(a,n)\sinh\left(\frac{\pi n\sqrt{|dD|}}{a}\right).
$$
\end{corollary}
Corollary \ref{fserexp} was obtained by Duke for $N=D=1$ (\cite{du}, Proposition 4) and Jenkins for $N=1$ and $D>1$ (\cite{jen}, Theorems 1.5 and 2.2). Choi, Jeon, Kang, and Kim \cite{cjkk} obtained an analogous formula for the subgroups $\Gamma_0(p)^+$ for $p$ prime, later generalized by Kang and Kim \cite{kk} to $\Gamma_0(N)^+$ for arbitrary $N$. \\

Next we examine algebraicity properties of the Fourier coefficients of $f^*_{d,D,N}$. Bengoechea \cite{ben} showed that for $\Delta<0$ and $k\in\{2,3,4,5,7\}$ (so that $S_{2k}=\{0\}$), the Fourier coefficients of $f_{k,\Delta}$ lie in the Hilbert class field of $\Q(\sqrt{\Delta})$. We have the following extension to the weight $2$ case.

\begin{theorem}\label{inte}
If $\Gamma_0(N)$ has genus $0$, then the Fourier coefficients of the meromorphic part of $f^*_{d,D,N}$ are real algebraic integers in the field $\Q\left(\sqrt{D}\right)$.
\end{theorem}

Eventually, we compute regularized inner products of meromorphic analogues of the $f^*_{d,D,N}$. For this we restrict to the case $D=N=1$ and consider the meromorphic modular forms
$$
f_{d}(\tau) := f^*_{d,1,1}(\tau) - H(d,1,1)E_2 ^*(\tau).
$$
The usual inner product $\langle f_{d},f_{\delta}\rangle$ for negative discriminants $d, \delta$ does not converge, so we need to use a regularization by Bringmann, Kane, and von Pippich. Moreover, since the $f_{d}$ do not decay like cusp forms towards $i\infty$, we also have to apply Borcherds's regularization near the cusp $i\infty$ (see Section \ref{4} for a precise definition). We obtain the following evaluations, where $J(z):=j_{1,1}(z)-24$ denotes the {\it normalized modular $j$-invariant}. 

\begin{theorem}\label{rip}
Let $d$ be a negative discriminant and $\mathscr{Q}_{d}:=\mathscr{Q}_{d,1}$. 
\begin{itemize}
\item[(i)] If $\delta < d$ is another negative discriminant such that $\frac{\delta}{d}$ is not a square, then
$$
\left\langle f_{d}, f_{\delta}\right\rangle =  \frac{1}{2\pi}\sum_{Q\in\mathscr{Q}_{d}/\SL_2(\Z)\atop \mathcal{Q}\in\mathscr{Q}_{\delta}/\SL_2(\Z)}\frac{1}{w_{Q}w_{\mathcal{Q}}}\log\left|J(z_Q)-J(z_\mathcal{Q})\right|.
$$
\item[(ii)] If neither $-\frac{d}{3}$ nor $-\frac{d}{4}$ is a square, then
$$
\left\langle f_{d}, f_{d}\right\rangle =  \frac{1}{2\pi}\sum_{Q\in\mathscr{Q}_{d}/\SL_2(\Z)}\log\left|\sqrt{|d|}\frac{J'(z_Q)}{Q(1,0)}\right|+ 
\frac{1}{2\pi}\sum_{Q,\mathcal{Q}\in\mathscr{Q}_{d}/\SL_2(\Z)\atop Q\neq \mathcal{Q}}\log\left|J(z_Q)-J(z_\mathcal{Q})\right|.
$$
\item[(iii)] We have
$$
\left\langle f_{-3}, f_{-3}\right\rangle =\frac{1}{18\pi}\log\left|\frac{\sqrt{3}}{2}J^{'''}\left(\frac{1+i\sqrt{3}}{2}\right)\right| \quad \text{and}\quad \left\langle f_{-4}, f_{-4}\right\rangle =\frac{1}{8\pi}\log\left|2J^{''}(i)\right|.
$$
\end{itemize}
\end{theorem}

Note that $\log\left|J(z)-J(\z)\right|$ is a {\it Green's function} on the modular curve $X_0(1)$. The double traces over CM-values of Green's functions occurring in Theorem \ref{rip} have been related to heights of Heegner points on modular curves by Gross and Zagier \cite{gz}. Since the $f_{d}$ are modular forms of weight $2$, it would be enlightening to find a geometric interpretation of their inner products and see how they relate to height functions. In higher weight, Bringmann, Kane, and von Pippich \cite{bkcyc} wrote regularized inner products of the functions $f_{k,\Delta}$ for $\Delta<0$ in terms of double traces over CM-values of {\it higher Green's functions}, so we can see Theorem \ref{rip} as an extension of their result to the weight $2$ case. \\

The paper is organized as follows: In Section \ref{2}, we introduce the necessary notation and definitions. 
In Section \ref{3}, we prove Theorems \ref{anafexp}, \ref{fexp}, and \ref{inte}. Eventually, in Section \ref{4}, we compute the regularized inner products, proving Theorem \ref{rip}.

\section*{Acknowledgements}
We thank Kathrin Bringmann, Stephan Ehlen, and Markus Schwagenscheidt for valuable advice on writing the paper. We are also grateful to John Duncan, Michael Mertens, Ken Ono, Dillon Reihill, Larry Rolen, Tonghai Yang, Shaul Zemel, and the referee for helpful comments and conversations.

\section{Definitions and Preliminaries}\label{2}

\subsection{General notation}\label{nota}
Throughout this paper, 
we denote variables in the complex upper half-plane $\H$ by $\tau$, $z$, and $\varrho$ with $v:=\im(\tau), y:=\im (z), \eta:=\im(\varrho)$ and for $w\in\C$ we write $e(w):=e^{2\pi iw}$. For a matrix $M=\left(\begin{smallmatrix} a&b\\ c&d\end{smallmatrix}\right)\in\Gamma_0(N)$ and $\tau\in\H$, we set 
$$
M\tau := \frac{a\tau +b }{c\tau +d}\quad\text{ and }\quad j(M,\tau):=c\tau +d.
$$
For each point $\varrho\in\H$ we let $\Gamma_{N,\varrho}$ denote the stabilizer of $\varrho$ in $\Gamma_0(N)$ and set $w_{N,\varrho} := \frac{1}{2}\#\Gamma_{N,\varrho}$. Note that if $\rho:=\frac{1+i\sqrt{3}}{2}$ denotes the sixth order root of unity in $\H$, then we have
$$
w_\varrho:=w_{1,\varrho}=\begin{cases}
3, & \text{if $\varrho\in\SL_2(\Z)\rho$,} \\
2, & \text{if $\varrho\in\SL_2(\Z)i$,} \\
1, & \text{otherwise.}
\end{cases}
$$

 Furthermore, we define the divisor sum function $\sigma(m):=\sum_{d|m}d$ and the weight $2$ Eisenstein series
$$
E_2(\tau) := 1-24\sum_{m\ge 1}\sigma(m)e(m\tau),
$$ 
as well as its non-holomorphic completion
$$
E_2 ^* (\tau):= -\frac{3}{\pi v} + E_2(\tau),
$$
which transforms like a weight $2$ modular form for $\SL_2(\Z)$. In general, we will use a star to denote non-holomorphic modular forms (cf. Proposition \ref{anacontin} and Definition \ref{fdef}).

\subsection{Quadratic forms and traces of singular moduli}\label{2.2}

We denote an integral binary quadratic form $Q(X,Y)=aX^2+bXY+cY^2\in\Z[X,Y]$ by $Q=[a,b,c]$. The group $\SL _2(\Z)$ acts on the set of binary quadratic forms via 
\begin{equation}\label{act}
\left(Q\circ {\left(\begin{smallmatrix} \alpha&\beta\\ \gamma&\delta\end{smallmatrix}\right)} \right)(X,Y):=Q(\alpha X+\beta Y, \gamma X+\delta Y),
\end{equation}
leaving the discriminant $\Delta=b^2-4ac$ invariant. For a positive integer $N$ and a discriminant $\Delta$, we write $\mathscr{Q}_{\Delta,N}$ for the set of all binary integral quadratic forms $Q=[a,b,c]$ of discriminant $\Delta$ with $a>0$ and $N|a$. Then the group $\Gamma_0(N)$ acts on $\mathscr{Q}_{\Delta,N}$. For $\Delta<0$ and $Q\in\mathscr{Q}_{\Delta,N}$, we denote by $z_Q$ the {\it Heegner point of $Q$}, which is the unique zero of $Q(\tau,1)$ in $\H$. \\

For $\Delta <0$, we consider a splitting $\Delta=d\cdot D$ into a discriminant $d$ and a fundamental disriminant $D$ that are both congruent to squares modulo $4N$ (meaning that $d,D\equiv 0$ or $1\pmod{4}$ and $D$ is not a proper square multiple of an integer congruent to $0$ or $1\pmod{4}$) and denote by $\chi_D$ the {\it generalized genus character} corresponding to the decomposition $\Delta=d\cdot D$ as defined in \cite{gkz}.
%$$
%\chi_D(Q)=\chi_D([a,b,c]):=\begin{cases} \left(\frac{D}{r}\right), & \text{if $(a,b,c,D)=1$, where $Q$ represents $r$ and $(r,D)=1$,} \\
%0, & \text{if $(a,b,c,D)>1$.}
%\end{cases}
%$$
%One can show that $\chi_D$ is $\SL_2(\Z)$-invariant and does not depend on the choice of $r$ represented by $Q$ with $(r,D)=1$. 

\begin{definition}\label{tracedef}
For a $\Gamma_0(N)$-invariant function $g:\H\rightarrow\C$, we define the \emph{twisted trace of singular moduli of discriminants $d$ and $D$ of $g$} as
$$
\operatorname{tr}_{d,D,N} (g):= \sum_{Q\in\mathscr{Q}_{dD,N}/\Gamma_0(N)}\frac{\chi_D(Q)}{w_{N, Q}}g(z_Q),
$$
where $w_{N, Q}:=w_{N,z_Q}$. Moreover, we call
$$
H(d,D,N):=\operatorname{tr}_{d,D,N}(1)=\sum_{Q\in\mathscr{Q}_{dD,N}/\Gamma_0(N)}\frac{\chi_D(Q)}{w_{N, Q}} 
$$
the {\it Hurwitz class number of discriminants $d$ and $D$ and level $N$}. 
\end{definition}

\subsection{Polar Harmonic Maass Forms}\label{polar}

Now we define polar harmonic Maass forms and study their elliptic expansions, which we will need to compute the regularized inner products in Section \ref{4}. See \cite{bfor}, Section 13.3 for an introduction to polar harmonic Maass forms and their applications. 

\begin{definition}\label{PolarHMFDefn}
For $k\in\mathbb{Z}$, a {\it polar harmonic Maass form of weight $k$ for $\Gamma_0(N)$} is a continuous function $F\colon\mathbb{H}\to\mathbb{C}\cup\{\infty\}$ which is real-analytic outside a discrete set of points and satisfies the following conditions: 
\begin{itemize}
\item[i)] For every $M\in\Gamma_0(N)$ and $\tau\in\H$, we have 
$$
F(M\tau)=j(M,\tau)^k F(\tau).
$$
\item[ii)]
The function $F$ is annihilated by the \begin{it}weight $k$ hyperbolic Laplacian\end{it} 
$$
\Delta_{k}:= -v^2\left(\frac{\partial^2}{\partial u^2} + \frac{\partial^2}{\partial v^2}\right)+i kv \left(\frac{\partial}{\partial u} +i\frac{\partial}{\partial v}\right).
$$
\item[iii)] For every $z\in\mathbb{H}$, there exists an $n\in\mathbb{N}_0$ such that $(\tau-z)^nF(\tau)$ is bounded in some neighborhood of $z$.
\item[iv)] The function $F$ grows at most linearly exponentially at the cusps of $\Gamma_0(N)$.
\end{itemize} 
We denote by $\mathcal{H}_{k}(N)$ the space of weight $k$ polar harmonic Maass forms for $\Gamma_0(N)$.
\end{definition}

Polar harmonic Maass forms have {\it elliptic expansions} around every point $\varrho\in\H$, which converge if 
\begin{equation}\label{xdef}
X_\varrho(\tau) := \frac{\tau-\varrho}{\tau-\overline{\varrho}}
\end{equation}
is sufficiently small. These can be seen as counterparts to the more common $q$-series expansions at the cusps and also break into two pieces. 
 
 \begin{prop}[Proposition 2.2 of \cite{bk}, see also Subsection 2.3 of \cite{bkcyc}]\label{EllipticExpansion}
A  polar harmonic Maass form $F$ of weight $k\leq 0$ has an expansion around each point $\varrho\in\mathbb H$ of the form $F=F^+_{\varrho}+F^-_{\varrho}$,
where the {\it meromorphic part} $F^+_{\varrho}$ is given by
\begin{equation}\label{mp}
F^+_{\varrho}(\tau):=(\tau-\overline{\varrho})^{-k}\sum_{n\gg-\infty}A_{F,\varrho}^+(n)X_{\varrho}^n(\tau)
\end{equation}
and the {\it non-meromorphic part} $F^-_{\varrho}$ by
\begin{equation}\label{nmp}
F^-_{\varrho}(\tau):=(\tau-\overline{\varrho})^{-k}\sum_{n\ll\infty}A_{F,\varrho}^-(n)\beta_0\left(1-|X_{\varrho}(\tau)|^2;1-k,-n\right)X_{\varrho}^n(\tau)
.
\end{equation}
These expressions converge for $|X_{\varrho}(\tau)|\ll 1$.
Here, we have that
$$
\beta_0\left(w; a,b\right):=\beta\left(w; a,b\right)-\mathcal{C}_{a,b} \hspace{7mm} \text{ with }\hspace{7mm}
\mathcal{C}_{a,b}:=\sum_{\substack{0\leq j\leq a-1\\ j\neq -b}} \binom{a-1}{j}\frac{(-1)^j}{j+b},
$$
where the \emph{incomplete $\beta$-function} is defined by
$
\beta({w};a,b):=\int_0^{w} t^{a-1} (1-t)^{b-1} dt.
$
\end{prop}

\noindent We refer to the terms 
in \eqref{mp} and \eqref{nmp} 
which grow as $\tau\to \varrho$ as the \begin{it}principal part of $F$ around $\varrho$\end{it}. 

\begin{remark}
The hyperbolic Laplacian splits as
\begin{equation}\label{xidef}
\Delta_k=-\xi_{2-k}\circ \xi_k, \quad\text{where}\quad \xi_{k}:=2iv^{k} \overline{\frac{\partial}{\partial \overline{\tau}}}.
\end{equation}
 If $F$ satisfies weight $k$ modularity, then $\xi_{k}(F)$ is modular of weight $2-k$. Moreover, $\xi_k$ annihilates the meromorphic part of a polar harmonic Maass form, so that it maps weight $k$ polar harmonic Maass forms to weight $2-k$ meromorphic modular forms and its kernel is given by the space of weight $k$ meromorphic modular forms.  
\end{remark}

\subsection{Niebur-Poincar\'e series}

Here we introduce {\it Niebur-Poincar\'e series} and give their explicit Fourier expansion.
Following \cite{nie}, we define for $n > 0$ 
$$
F_{N, -n, s}(z):=2\pi \sqrt{n}\sum_{M\in\Gamma_\infty\backslash\Gamma_0(N)} e(-n\re (Mz))\im (Mz)^{\frac12}I_{s-\frac12}(2\pi n\im (Mz)),
$$
where $I_{s-\frac12}$ denotes the {\it $I$-Bessel function} and $\Gamma_\infty:=\left\{\pm\left(\begin{smallmatrix}
	1 & n\\
	0 & 1
\end{smallmatrix}\right):n\in\Z\right\}$.
This series converges absolutely and locally uniformly for $\re (s) >1$. The function $F_{N,-n,s}$ is a $\Gamma_0(N)$-invariant eigenfunction of the hyperbolic Laplacian with eigenvalue $s(1-s)$. Niebur showed that $F_{N,-n,s}$ is analytic in $s$ and has an analytic continuation to $s=1$.

\begin{prop}[Theorem 1 of \cite{nie}]\label{niebth}
	The function $F_{N, -n, s}$ has an analytic continuation $j_{N, n}$ to $s=1$, and $j_{N, n}\in \mathcal H_0 (N)$. It has a Fourier expansion of the form
$$
		j_{N, n}(z)=e(-nz)-e(-n\overline{z})+c_N(n,0)+\sum_{m\ge1}\left(c_N(n,m)e(mz)+c_N(n,-m)e(-m\overline{z})\right).
$$
The coefficients are given by
$$
c_N(n,m):= 2\pi \sqrt{n}\sum_{\substack{c\geq 1 \\ N|c}}\frac{K(m,-n;c)}{c}\times\begin{cases}\frac{1}{\sqrt{m}} I_1\left(\frac{4\pi\sqrt{mn}}{c}\right), & \text{if $m> 0$,}\\
\frac{2\pi \sqrt{n}}{c},& \text{if $m= 0$,}\\
\frac{1}{\sqrt{|m|}} J_1\left(\frac{4\pi\sqrt{|m|n}}{c}\right), & \text{if $m< 0$,}
\end{cases}
$$
where 
$$
	K(m,n;c):=\sum_{\substack{a, d\pmod c \\ ad\equiv 1\pmod c }} e\left(\frac{m d +na}{ c}\right)
$$
denotes the \emph{Kloosterman sum} and $I_1$, $J_1$ are the \emph{first order $I$- and $J$-Bessel functions}, respectively.
\end{prop}

The constants $c_N(n,0)$ can be explicitly evaluated. For example, for $N=1$ we obtain
$$
c_{1} (n,0) = 24\sigma(n)
$$
and for $p$ prime we have
\begin{equation}
c_{p}(n,0)=-\frac{24}{p^2 -1}\left(\sigma(n)-p^2 \sigma\left(\frac{n}{p}\right)\right),
\label{cm0eq}
\end{equation}
where $\sigma(\ell):=0$ if $\ell\notin\Z$ (see \cite{cjkk} for a similar calculation).

\subsection{Petersson's Poincar\'e series}

For $w,s\in\C$, $w\neq 0$, we let
\begin{equation}
\phi_s(w):=w^{-1}|w|^{-s}
\label{phi}
\end{equation}
and define
\begin{multline}\label{pdefeq}
H_{N,s}(z,\tau):=-\frac{v^s}{2\pi}\sum_{M\in\Gamma_0(N)}\phi_s\left(j(M, \tau)^2\frac{(M\tau- z)(M\tau -\overline{z})}{y}\right)\\
=-\frac{v^s}{2\pi}\sum_{M\in\Gamma_0(N)}\phi_s\left(\frac{(\tau- Mz)(\tau -M\overline{z})}{\im (Mz)}\right).
\end{multline}

These sums converge locally uniformly for $\re(s)>0$ and define analytic functions in $s$. They satisfy modularity of weight $0$ in $z$ and of weight $2$ in $\tau$. Bringmann and Kane showed that they have an analytic continuation $H^*_{N}$ to $s=0$, which are a polar harmonic Maass forms of weight $2$ with simple poles at $\Gamma_0(N)$-equivalent points to $z$. For this, they used a splitting of the sum due to Petersson and obtained an analytic continuation of the Fourier expansion of every part by Poisson summation and locally uniform estimates. 

\begin{prop}[Lemma 4.4 of \cite{bk}]\label{anacontin}
The function $H_{N,s}$ has an analytic continuation $H^*_{N}$ to $s=0$. We have 
$$
z\mapsto H_{N}^*(z,\tau)\in \mathcal{H}_0(N) \quad \text{and}\quad \tau\mapsto H_{N}^*(z,\tau)\in \mathcal{H}_2(N).
$$
 Furthermore, the function
$$
\tau\mapsto H_{N}(z,\tau):= H_{N}^*(z,\tau) + \frac{1}{[\SL_2(\Z):\Gamma_0(N)]}E_2^*(\tau)
$$
is a meromorphic modular form of weight $2$ for $\Gamma_0(N)$ with only simple poles at points that are $\Gamma_0(N)$-equivalent to $z$.
\end{prop}

\begin{remark}
Note that $H_{N}^*(z,\cdot)$ has principal part $-\frac{w_{N,z}}{2\pi i }\frac{1}{\tau-z}$ at $\tau=z$.
\end{remark}

The Fourier coefficients of $H^*_{N}(z,\cdot)$ were computed in \cite{bklor}, where it was shown that they are given by the Niebur-Poincar\'e series $j_{N, n}$ of Proposition \ref{niebth}, evaluated at $z$.

\begin{prop}[Theorem 1.1 of \cite{bklor}]\label{pexp}
For $v>\operatorname{max}\left\{y, \frac{1}{y}\right\}$, we have 
$$
H_{N}^*(z,\tau)=\frac{3}{\pi [\SL_2(\Z):\Gamma_0(N)] v}+\sum_{n>0}j_{N, n}(z)e(n\tau).
$$
\end{prop}

\section{Weight $2$ Modular Forms Associated to Imaginary Quadratic Fields}\label{3}

Now we define and study the weight $2$ analogues of the functions $f_{k,\Delta}$.

\begin{definition}\label{fdef}
For $N\in\N$, discriminants $d$, $D$ that are congruent to squares modulo $4N$ with $D$ fundamental and $dD$ negative, and $s\in\C$ with $\re(s)>0$, we let
$$
f_{d,D,N,s}(\tau):=\frac{(|dD|)^{\frac{1+s}{2}}v^s}{2^{1+s}\pi}\sum_{Q\in\mathscr{Q}_{dD,N}}\chi_D(Q)\phi_s(Q(\tau, 1))
$$
with $\phi_s$ as in \eqref{phi} and define $f^*_{d,D,N}$ to be the analytic continuation of $f_{d,D,N,s}$ to $s=0$.
\end{definition} 

\begin{remark}
The existence of the analytic continuation is established by combining Lemma \ref{ftrp} with the analytic continuation of $H_{N,s}$ stated in Proposition \ref{anacontin}.
\end{remark}

With the trace operation from Definition \ref{tracedef}, we obtain the following relation. 

\begin{lemma}\label{ftrp}
We have
$$
f_{d,D,N,s}(\tau)=-\operatorname{tr}_{d,D,N}\left(H_{N,s} (\cdot,\tau)\right).
$$
\end{lemma}

\begin{proof}
For $M\in\Gamma_0(N)$, $Q\in \mathscr{Q}_{dD,N}$, and the group action defined in \eqref{act}, we have
$z_{Q\circ M} = M^{-1}z_Q$ and
$$
Q(\tau, 1)=\frac{\sqrt{|dD|}}{2}\frac{(\tau-z_Q)(\tau-\overline{z}_Q)}{\im (z_Q)},
$$
since $\im(z_Q)=\frac{\sqrt{|dD|}}{2a}$ for $Q=[a,b,c]$. Thus it follows
\begin{align*}
H_{N,s} (z_Q,\tau)&=-\frac{v^s }{2\pi}\sum_{M\in\Gamma_0(N)}\phi_s\left(\frac{(\tau-Mz_Q)(\tau-M\overline{z}_Q)}{\im (Mz_Q)}\right)\\
&= -\frac{v^s }{2\pi}\sum_{M\in\Gamma_0(N)}\phi_s\left(\frac{(\tau-z_{Q\circ M})(\tau-\overline{z}_{Q\circ M})}{\im (z_{Q\circ M})}\right)\\
&= -\frac{v^s }{2\pi}\sum_{M\in\Gamma_0(N)}\phi_s\left(\frac{2}{\sqrt{|dD|}}( Q \circ M)(\tau, 1)\right)\\
&=-\frac{v^s (|dD|)^{\frac{1+s}{2}}}{2^{2+s}\pi}\sum_{M\in\Gamma_0(N)}\phi_s\left(( Q \circ M)(\tau, 1)\right)
\end{align*}
Taking the twisted trace we obtain
\begin{multline*}
\operatorname{tr}_{d,D,N}\left(H_{N,s} (\cdot,\tau)\right)=-\frac{v^s (|dD|)^{\frac{1+s}{2}}}{2^{2+s}\pi}\sum_{Q\in\mathscr{Q}_{dD,N}/\Gamma_0(N)}\frac{\chi_D(Q)}{w_{N,Q}}\sum_{M\in\Gamma_0(N)}\phi_s\left(( Q \circ M)(\tau, 1)\right)\\
=-\frac{v^s (|dD|)^{\frac{1+s}{2}}}{2^{2+s}\pi}\sum_{Q\in\mathscr{Q}_{dD,N}}\frac{\chi_D(Q)}{w_{N,Q}}\cdot 2w_{N,Q}\phi_s\left( Q (\tau, 1)\right)=-f_{d,D,N,s}(\tau)
\end{multline*}
\end{proof}

Theorem \ref{fexp} now follows directly from Proposition \ref{pexp} and taking the analytic continuation to $s=0$ in Lemma \ref{ftrp}.\\

We now move on to compute the Fourier expansion of $f^*_{d,D,N}$ directly and prove Theorem \ref{anafexp}.

\begin{proof}[Proof of Theorem \ref{anafexp}]
We follow the approach of Appendix 2 of \cite{zamrq}. For $v > \frac{\sqrt{|dD|}}{2}$, we obtain by Poisson summation
\begin{align*}
f_{d,D,N,s} (\tau)&= \frac{(|dD|)^{\frac{1+s}{2}}v^s}{2^{1+s}\pi}\sum_{a>0 \atop N|a } \sum_{ b\in\Z \atop b^2\equiv dD\pmod{4a}}\chi_D\left(\left[a,b,\frac{b^2-dD}{4a}\right]\right)\phi_s \left(a\tau ^2+b\tau+\frac{b^2-dD}{4a}\right)\\
&=\frac{(|dD|)^{\frac{1+s}{2}}v^s}{2^{1+s}\pi}\sum_{a>0 \atop N|a }\sum_{n\in\Z} \sum_{ b\pmod{2a} \atop b^2\equiv dD\pmod{4a}}\chi_D\left(\left[a,b,\frac{b^2-dD}{4a}\right]\right)\\
&\qquad\qquad\times\int_\R\phi_s \left(a(\tau +t) ^2+b(\tau+t)+\frac{b^2-dD}{4a}\right)e(-nt)dt.
\end{align*}
Here we used that
$$
a\tau ^2+(b+2an)\tau+\frac{(b+2an)^2-dD}{4a} = a(\tau +n) ^2+b(\tau+n)+\frac{b^2-dD}{4a}
$$
and that $\chi_D$ is invariant under translation.
Together with
%\begin{multline*}
$$
\int_\R\phi_s \left(a(\tau +t) ^2+b(\tau+t)+\frac{b^2-dD}{4a}\right)e(-nt)dt = a^{-1-s}e(n\tau)\int_{\R+iv}\phi_s \left(t ^2-\frac{dD}{4a^2}\right)e(-nt)dt
%\end{multline*}
$$
\noindent we obtain 
$$
f_{d,D,N,s} (\tau)
=\frac{(|dD|)^{\frac{1+s}{2}}v^s}{2^{1+s}\pi}\sum_{a>0 \atop N|a }\sum_{n\in\Z} S_{d,D}(a,n) a^{-1-s}e(n\tau)\int_{\R+iv}\phi_s \left(t ^2-\frac{dD}{4a^2}\right)e(-nt)dt.
$$

First we consider terms with $n\neq 0$. 
We have to show locally uniform convergence in $s$ of the double sum. For this we will bound the integral locally uniformly for $\sigma:=\re(s) > -\varepsilon$ for some $\varepsilon >0$ and independently of $a$ and $n$. First we write 
$$
%\int_{\R +iv}\phi_s\left(t^2+\frac{dD}{4a^2}\right)e(-nt)dt = \int_\R \left((t+iv)^2+\frac{dD}{4a^2}\right)^{-1}\left(\left(t^2-v^2+\frac{dD}{4a^2}\right)^2 + 4v^2 t^2\right)^{-\frac{s}{2}}e(-nt)dt.
\int_{\R +iv}\phi_s\left(t^2-\frac{dD}{4a^2}\right)e(-nt)dt = \int_\R \left(\left(t^2-v^2-\frac{dD}{4a^2}\right)^2 + 4v^2 t^2\right)^{-\frac{s}{2}}\frac{e(-nt)dt}{(t+iv)^2-\frac{dD}{4a^2}}.
$$
Note that the integrand is holomorphic in $t$ in the region $\im (t) > \frac{\sqrt{|dD|}}{2}-v$. Thus for $n<0$, we may shift the path of integration to $\R + i\infty$ and the integral vanishes. 

For $n>0$, we may fix $\alpha \in \left(0, v-\frac{\sqrt{|dD|}}{2}\right)$ and shift the path of integration to $\R -i\alpha$. This yields
\begin{multline*}
\left|\int_{\R - i\alpha }\phi_s\left((t+iv)^2-\frac{dD}{4a^2}\right)e(-nt)dt \right| \\
\leq 2 e^{-2\pi n\alpha}\int_0 ^\infty \left(\left(t^2-(v-\alpha) ^2 -\frac{dD}{4a^2}\right)^2 + 4(v-\alpha )^2 t^2\right)^{-\frac{1+\sigma}{2}} dt.
\end{multline*}

Now we apply the estimates
$$
\left(t^2-(v-\alpha )^2-\frac{dD}{4a^2}\right)^2 + 4(v-\alpha )^2 t^2 \geq \begin{cases}
 \left((v-\alpha )^2 +\frac{dD}{4}\right)^2, & \text{for every $t$, } \\
\left(t^2+(v-\alpha) ^2 \right)^2, & \text{for $t>v-\alpha$,}
\end{cases}
$$
to obtain
\begin{multline*}
\int_0 ^\infty \left(\left(t^2-(v-\alpha )^2-\frac{dD}{4a^2}\right)^2 + 4(v-\alpha )^2 t^2\right)^{-\frac{1+\sigma}{2}} dt \\
\leq \int _0 ^{v-\alpha} \left((v-\alpha )^2 +\frac{dD}{4}\right)^{-1-\sigma} dt
+ \int_{v-\alpha} ^\infty \left(t^2+(v-\alpha) ^2 \right)^{-1-\sigma} dt.
\end{multline*}
The last bound is locally uniform for $\sigma > -\frac12$ and independent of $a$ and $n$. Thus the overall sum is uniformly bounded by
$$
\ll \sum_{n\geq 1}\sum_{a>0\atop N|a}S_{d,D}(a,n)a^{-1-\sigma}e^{-2\pi n\alpha}.
$$

We define the half-integral weight Kloosterman sum as
$$
K^*(m,n,c):=\sum_{a,d\pmod{c}^*\atop ad\equiv 1\pmod{c}}\left(\frac{c}{d}\right)\left(\frac{-4}{d}\right)^{3/2} e\left(\frac{na+md}{c}\right),
$$
where $(\frac{\cdot}{\cdot})$ denotes the Kronecker symbol. Plugging $c\mapsto 4a$ into Proposition 3 of \cite{dit} and noting that the definition of $S$ given there differs from ours by a factor $2$, we obtain 
$$
S_{d,D}(a,n)=\frac{1-i}{4}\sum_{r|(a,n)}\left(\frac{D}{r}\right)\sqrt{\frac{r}{a}}\left(1+\left(\frac{4}{a/r}\right)\right)K^*\left(d,\frac{n^2 D}{r^2}, \frac{4a}{r}\right),
$$
and hence
$$
\sum_{a>0\atop N|a}S_{d,D}(a,n)a^{-1-\sigma}
=\frac{1-i}{4}\sum_{r|n} r^{-1-\sigma}\left(\frac{D}{r}\right)
\sum_{a>0 \atop N|a}\left(1+\left(\frac{4}{a}\right)\right)\frac{K^*\left(d,\frac{n^2 D}{r^2},4a\right)}{a^{\frac32+\sigma}}.
$$
It has been observed in the remark following Theorem 2.1 of \cite{fo} that the Selberg-Kloosterman zeta function
$$
S_{m,n}(s):=\sum_{a>0}\frac{K^*\left(m,n,a\right)}{a^s}
$$
has an analytic continuation to $s=\frac32$ for $mn<0$. 
Since $S_{m,n}$ has only finitely many poles in $[1,2]$, there is an $\varepsilon > 0$ such that the function $S_{d,\frac{n^2 D}{r^2}}\left(\frac32 +\sigma\right)$ has an analytic continuation to $\sigma > -\varepsilon$. This gives a locally uniform bound for $\sigma > -\varepsilon$ and we obtain the analytic continuation for the sum over the positive $n$ by just plugging in $s=0$.\\

Now we have, for $v>\frac{\sqrt{|dD|}}{2}$ and $n>0$,
$$
\int_{\R+iv}\left(t ^2-\frac{dD}{4a^2}\right)^{-1}e(-nt)dt = -\frac{4\pi  a}{\sqrt{|dD|}} \sinh\left(\frac{\pi n \sqrt{|dD|}}{a}\right),
$$
since $t\mapsto \sinh(\kappa t)$ is the inverse Laplace transform of $s\mapsto \frac{\kappa}{s^2 -\kappa^2}$ (see for example (29.3.17) of \cite{as}). So all in all we obtain that for $n>0$, the $n$-th Fourier coefficient of $f_{d,D,N}^*$ equals
$$
-2\sum_{a>0 \atop N|a }S_{d,D}(a,n)\sinh\left(\frac{\pi n\sqrt{|dD|}}{a}\right).
$$
Finally, it follows from Proposition \ref{pexp} and Lemma \ref{ftrp} that the remaining part of the Fourier expansion, i.e.~the $n=0$ term, equals 
$$
-\operatorname{tr}_{d,D,N}\left(\frac{3}{\pi[\SL_2(\Z):\Gamma_0(N)]v}\right)=-\frac{3H(d,D,N)}{\pi[\SL_2(\Z):\Gamma_0(N)]v}.
$$
%One can also obtain this term by a direct but lengthy calculation which we omit here.
\end{proof}

\begin{proof}[Proof of Theorem \ref{inte}]
By Theorem \ref{fexp}, the $n$-th Fourier coefficient of $f^*_{d,D,N}$ is $-\operatorname{tr}_{d,D,N}(j_{N,n} )$.
If $\Gamma_0(N)$ has genus $0$, then $j_{N,n}$ is weakly holomorphic on the modular curve $X_0(N)$. Lemma 5.1 (v) of \cite{bruo} states that the twisted Heegner divisor 
$$
Z_{d,D,N}:= \sum_{Q\in\mathscr{Q}_{dD,N}/\Gamma_0(N)}\frac{\chi_D(Q)}{w_{N, Q}}z_Q
$$
is defined over $\Q(\sqrt{D})$. This means that
$$
\left\langle Z_{d,D,N}, j_{N,n}\right\rangle :=\sum_{Q\in\mathscr{Q}_{dD,N}/\Gamma_0(N)}\frac{\chi_D(Q)}{w_{N, Q}}j_{N,n}(z_Q) =\operatorname{tr}_{d,D,N}(j_{N,n})\in\Q(\sqrt{D}).
$$
By Theorem I of \cite{cy}, $j_{N,1}(z_Q)$ is an algebraic integer for every quadratic form $Q\in\mathscr{Q}_{dD,N}$. Now $j_{N,n}$ is a polynomial in $j_{N,1}$, so the twisted sum $\operatorname{tr}_{d,D,N}(j_{N,n} )$ is also an algebraic integer, which implies the statement.
\end{proof}

\section{Regularized Inner Products}\label{4}

In this section we restrict to the full modular group and therefore drop the subscript $N$ throughout. Let $f$, $g$ be meromorphic modular forms of weight $k$ which decay like cusp forms at $i\infty$ and have poles at $\mathfrak{z}_1,\dots,\mathfrak{z}_r\in \SL_2(\Z)\backslash\H$.  We choose a fundamental domain $\mathcal{F}$ such that for every $j\in\{1,\dots,r\}$, the representative of $\mathfrak{z}_j$ in $\mathcal{F}$ lies in the interior of $\Gamma_{\mathfrak{z}_j}\mathcal{F}$. We identify the $\mathfrak{z}_1,\dots,\mathfrak{z}_r\in \SL_2(\Z)\backslash\H$ with their representatives in $\mathcal{F}$.\\

For an analytic function $A(s)$ in $s=(s_1,\dots,s_r)$, denote by $\mathrm{CT}_{s=0}A(s)$  the constant term of the meromorphic continuation of $A(s)$ around $s_1=\cdots =s_{r}=0$.  Then the regularized inner product introduced in \cite{bkcyc} is given by
\begin{equation}\label{eqn:OurReg}
\left<f,g\right>:= \operatorname{CT}_{s=0}\left(\int_{\mathcal{F}} f(\tau) \prod_{\ell=1}^r\left|X_{\z_\ell}(\tau)\right|^{2s_\ell} \overline{g(\tau)} v^{k}\frac{dudv}{v^2}\right).
\end{equation}
 Note that, as $z\to \mathfrak{z}_\ell$, we have $X_{\z_\ell}(z)\to 0$, so the integral in \eqref{eqn:OurReg} converges if we have $\re (s_{\ell})\gg 0$ for every $1\leq \ell\leq r$. One can show that the regularization is independent of the choice of fundamental domain. Since the functions we integrate do not decay like cusp forms, we need to use another regularization by Borcherds \cite{bor}. Namely for holomorphic modular forms $f,g$ of weight $k$, we define
  \begin{equation}
\label{eq:Ifg} \left\langle f,g\right\rangle :=  \text{CT}_{s=0}\left(\lim_{T\rightarrow\infty}\int_{\mathcal{F}_T} f(\tau) \overline{g(\tau)}v^{k-s} \frac{dudv}{v^2}\right),
  \end{equation}
	whenever it exists. Here, for a fundamental domain $\mathcal{F}$ and $T>0$, we set 
	$$
	\mathcal{F}_T :=\left\{z\in\mathcal{F}: \im(z)\le T\right\}.
	$$
	To compute the inner products in Theorem \ref{rip}, we split the domain of integration into a part which contains all the poles of the integrands, where we apply the regularization of Bringmann, Kane, and von Pippich, and a part around the cusp $i\infty$, where we apply Borcherds's regularization. Therefore, the regularized integral will look like
\begin{multline}\label{intsplit}
\left\langle f,g\right\rangle =\text{CT}_{s=0}\left(\int_{\mathcal{F}_Y}f(\tau)\prod_{\ell=1}^r\left|X_{\z_\ell}(\tau)\right|^{2s_\ell}\overline{g(\tau)}v^k\frac{dudv}{v^2} \right)\\
+ \text{CT}_{s=0}\left(\lim_{T\rightarrow\infty}\int_{\mathcal{F}_T\setminus \mathcal{F}_Y} f(\tau)  \overline{g(\tau)}v^{k-s} \frac{dudv}{v^2}\right),
\end{multline}
where $f,g$ are meromorphic modular forms of weight $k$ and $Y>1$ a fixed constant, such that all poles of $f$ and $g$ lie in $\mathcal{F}_Y$. Here we can assume that $\mathcal{F}$ contains all the poles of $f$ and $g$ as well as $[0,1]+i[Y,\infty]$.\\

To prepare the proof, we first look at elliptic expansions of the polar harmonic Maass forms $H_z(\tau):=H_1(z,\tau)$. For $X_\varrho(\tau)\ll 1$, we have 
\begin{equation}\label{Hellxp}
H_z(\tau)=-\frac{1}{2\pi i}\frac{J'(\tau)}{J(\tau)-J(z)}=\frac{1}{(\tau-\overline{\varrho})^2}\left(-\frac{\delta_{z,\varrho}w_z y}{\pi }X_\varrho(\tau)^{-1}+\sum_{n\geq 0}a_{z,\varrho}(n)X_\varrho(\tau)^n\right),
\end{equation}
where $\delta_{z,\varrho}$ is defined to be $1$ if $z\in \SL_2(\Z)\varrho$ and $0$ otherwise (cf. the remark following Proposition \ref{anacontin}). Furthermore, let
$$
G_z(\tau):=-\frac{1}{2\pi }\log\left|J(\tau)-J(z)\right|,
$$
so that $\xi_0 (G_z)=H_z$ with $\xi_0$ as in \eqref{xidef}. 
%\begin{multline*}
%\xi_0 G_z(\tau) = -\frac{i}{\pi }\overline{\frac{\partial G_z}{\partial\overline{\tau}}(\tau)}\\
%= \frac{1}{2\pi i}\overline{\frac{\partial }{\partial\overline{\tau}}\left(\log\left(j(\tau)-j(z)\right)+\log\left(j(-\overline{\tau})-j(-\overline{z})\right)\right)} \\
%= \frac{1}{2\pi i}\overline{\frac{-j'(-\overline{\tau})}{j(-\overline{\tau})-j(-\overline{z})}}=-\frac{1}{2\pi i}\frac{j'(\tau)}{j(\tau)-j(z)}=H_z(\tau)
%\end{multline*} 

\begin{lemma}
The function $G_z$ is a weight $0$ polar harmonic Maass form. For every $\varrho\in\H$, it has an elliptic expansion
\begin{equation}\label{Gellxp}
G_z(\tau)=- \frac{\delta_{z,\varrho}w_z}{2\pi }\log(\left|X_\varrho(\tau)\right|)+\sum_{n\geq 0}A^+_{z,\varrho}(n)X_\varrho(\tau)^n  + \sum_{n> 0}A^-_{z,\varrho}(n)\overline{X_\varrho(\tau)}^{n},
\end{equation}
which converges for $|X_\varrho(\tau)|\ll 1$. 
Moreover, we have 
$$
A^-_{z,\varrho}(n)=\frac{\overline{a_{z,\varrho}(n-1)}}{4\eta n}
$$
with $a_{z,\varrho}(n)$ as in \eqref{Hellxp} and 
\begin{equation}\label{a0}
A^+_{z,\varrho}(0)=-\frac{1}{2\pi}\times\begin{cases}
\log\left|J(\varrho)-J(z)\right|,& \text{if $\varrho\notin\SL_2(\Z) z$,}\\
\log\left|2yJ'(z)\right|,& \text{if $\varrho\in \SL_2(\Z) z$ and $i,\rho\notin\SL_2(\Z) z$,}\\
 \log\left|2J''(i)\right|,& \text{if $\varrho, z\in\SL_2(\Z) i$,}\\
\log\left|\frac{\sqrt{3}}{2}J'''(\rho)\right|,& \text{if $\varrho, z\in\SL_2(\Z) \rho$.}
\end{cases}
\end{equation}
\end{lemma}

\begin{proof}
One easily checks that $G_z$ is a polar harmonic Maass form of weight $0$. By Proposition \ref{EllipticExpansion}, it has for every $\varrho\in\H$ an elliptic expansion 
$$
G_z(\tau)=\sum_{n\gg -\infty}A_{z,\varrho}^+(n)X_\varrho(\tau)^n  + \sum_{ n\ll \infty}A_{z,\varrho}^-(n)\beta_0\left(1-\left|X_\varrho(\tau)\right|^2, 1, -n\right)X_\varrho(\tau)^{n}
$$
for $X_\varrho(\tau)\ll 1$.
Noting that
$$
\beta_0(1-r^2; 1, -n)=\int_0 ^{1-r^2}(1-t)^{-n-1}dt +\delta_{n\neq 0}\cdot\frac{1}{n}= \begin{cases}
\frac{r^{-2n}}{n},& \text{if $n\neq 0$,} \\
-\log(r^2), &\text{if $n=0$,}
\end{cases}
$$
we obtain an elliptic expansion of the shape
$$
G_z(\tau)=A_{z,\varrho}^-(0)\log(\left|X_\varrho(\tau)\right|^2)+\sum_{n\gg -\infty}A^+_{z,\varrho}(n)X_\varrho(\tau)^n  + \sum_{n\ll \infty \atop n\neq 0}A_{z,\varrho}^-(n)\overline{X_\varrho(\tau)}^{-n}
$$
Now using 
$$
\xi_0 \left(\overline{X_\varrho (\tau)}\right)=2i\overline{\partial_{\overline{\tau}}\frac{\overline{\tau}-\overline{\varrho}}{\overline{\tau}-\varrho}}=-\frac{4\eta}{(\tau-\overline{\varrho})^2}
$$
and 
$$
\xi_0\left(\log(|X_\varrho(\tau)|^2)\right) =2i\overline{\partial_{\overline{\tau}}\log(|X_\varrho(\tau)|^2)} =2i\overline{\frac{\partial_{\overline{\tau}}|X_\varrho(\tau)|^2}{|X_\varrho(\tau)|^2}}
= 2i\overline{X_\varrho(\tau)}\frac{\overline{\partial_{\overline{\tau}}\overline{X_\varrho(\tau)}}}{|X_\varrho(\tau)|^2}=-\frac{4\eta}{(\tau-\overline{\varrho})^2}X_\varrho(\tau)^{-1}
$$
gives
$$ 
\xi_0 \left(G_z(\tau)\right)=-\frac{4\eta}{(\tau-\overline{\varrho})^2}\left(\overline{A_{z,\varrho}^-(0)}X_\varrho(\tau)^{-1}+\sum_{n\ll\infty\atop n\neq 0}n\overline{A_{z,\varrho}^-(n)}X_\varrho(\tau)^{-n-1}\right).
$$
We compare with $\eqref{Hellxp}$ and obtain $A_{z,\varrho}^-(n)=0$ for $n>0$, 
$$
A_{z,\varrho}^-(0)=-\frac{\delta_{z,\varrho}w_z}{2\pi },\quad\text{and}\quad-4\eta n\overline{A_{\varrho, z}(n)}=a_{\varrho, z}(-n-1).
$$
for $n< 0$. We also have $A^+_{z,\varrho}(n)=0$ for $n< 0$ since the principal part of $G_z$ at $\varrho$ comes entirely from the $n=0$ term.\\

For the evaluation of $A_{z,\varrho}^+(0)$, note that
$$
A^+_{z,\varrho}(0)=-\frac{1}{2\pi}\lim_{\tau\rightarrow\varrho}\left(\log\left|J(\tau)-J(z)\right|- \delta_{z,\varrho}w_z\log(\left|X_\varrho(\tau)\right|)\right),
$$
which equals 
$$
G_z(\varrho)=-\frac{1}{2\pi w_z}\log\left|J(\tau)-J(z)\right|
$$
if $\varrho\neq z$. If $\varrho=z$, note that 
\begin{multline*}
\lim_{\tau\rightarrow z}\left(\log\left|J(\tau)-J(z)\right|- w_z\log(\left|X_\varrho(\tau)\right|)\right) 
= \lim_{\tau\rightarrow z}\left(\log\left|J(\tau)-J(z)\right|+ \log\left(\left|\frac{\tau -\overline{z}}{\tau -z}\right|^{w_z}\right)\right)\\
=\lim_{\tau\rightarrow z}\log\left|\frac{J(\tau)-J(z)}{(\tau -z)^{w_z}}\right| + \log\left((2y)^{w_z}\right)
=\log\left|\frac{(2y)^{w_z}}{w_z!}J^{(w_z)}(z)\right|,
\end{multline*}
which implies the statement.
\end{proof}

\begin{lemma}\label{hprod}
For every $z,\z\in\H$, we have
$$
\left\langle H_z, H_\z\right\rangle = -A^+_{\z,z}(0)
$$
with $A^+_{\z,z}(0)$ as in \eqref{a0}.
\end{lemma}

\begin{proof}
%Let $Y>1$ be a fixed constant, such that all poles of $H_z$ and $H_\z$ have imaginary part less than $Y$. We can assume that $\mathcal{F}$ contains $z$, $\z$, and $[0,1]+i[Y,\infty]$ Then we split the regularized integral as 
%\begin{multline}\label{intsplit}
%\left\langle H_z, H_\z\right\rangle =\text{CT}_{(s_1,s_2)=(0,0)}\int_{\mathcal{F}_Y}H_z(\tau)|X_z(\tau)|^{2s_1}|X_\z(\tau)|^{2s_2}\overline{H_\z(\tau)}dudv
%+ \text{CT}_{s=0}\int_{\mathcal{F}\setminus \mathcal{F}_Y}H_z(\tau)\overline{H_\z(\tau)}v^{-s}dudv
%\end{multline}

Applying Stokes's Theorem to the second summand of \eqref{intsplit}, we obtain for $\re(s)>1$
\begin{align*}
\int_{\mathcal{F}_T\setminus \mathcal{F}_Y} H_z(\tau)&\overline{H_\z(\tau)}v^{-s}dudv =-\int_{\mathcal{F}_T\setminus \mathcal{F}_Y}H_z(\tau) \xi_0\left(\overline{G_\z(\tau)}\right)v^{-s}dudv\\ 
&=-\int_{\mathcal{F}_T\setminus \mathcal{F}_Y}\xi_0\left(\overline{H_z(\tau)G_\z(\tau)v^{-s}}\right)dudv - s\int_{\mathcal{F}_T\setminus \mathcal{F}_Y} H_z(\tau)G_\z(\tau)v^{-s-1}dudv\\
&=-\int_{\partial(\mathcal{F}_T\setminus \mathcal{F}_Y)}H_z(\tau)G_\z(\tau)v^{-s}d\tau- s\int_{\mathcal{F}_T\setminus \mathcal{F}_Y} \overline{H_\z(\tau)G_z(\tau)}v^{-s-1}dudv\\
&=\int_0 ^1 H_z(u+iT)G_\z(u+iT)T^{-s}du-\int_0 ^1 H_z(u+iY)G_\z(u+iY)Y^{-s}du \\
&\qquad + s\int_Y ^T \int_0 ^1  \overline{H_z(u+iv)G_\z(u+iv)}v^{-s-1}dudv.
\end{align*}

Now we have $G_\z(\tau)=v+O(1)$ and $H_z(\tau)=O(1)$ as $v\rightarrow\infty$. Hence, the first summand vanishes as $T\rightarrow \infty$, provided that $\re(s)$ is sufficiently large, and the limit as $T\rightarrow \infty$ of the third integral has a meromorphic continuation to $\C$ with the only pole at $s=1$. Therefore the contributions of the first and third summand vanish in the analytic continuation to $s=0$.
%the function
%$$
%M(v,s):= \int_0 ^1 H_z(u+iv)G_\z(u+iv)v^{-s}du 
%$$
%grows like $v^{1-s}$ as $v\rightarrow\infty$. Hence for $\re(s)>1$, we have 
%$$
%\lim_{T\rightarrow \infty}M(T,s)=0 \quad \text{and} \quad \lim_{T\rightarrow\infty} \int_Y ^T M(v,s+1)dv = \frac{Y^{1-s}}{1-s} + o(1)
%$$
 %Therefore,  
The second summand is analytic in $s=0$, so in total we have
\begin{equation}\label{Yeq}
\text{CT}_{s=0}\int_{\mathcal{F}\setminus \mathcal{F}_Y}H_z(\tau)\overline{H_\z(\tau)}v^{-s}d\tau = -\int_0 ^1 H_z(u+iY)G_\z (u+iY)du.
\end{equation}

%The integral 
%$$
%\lim_{T\rightarrow\infty}\int_{\mathcal{F}_T\setminus \mathcal{F}_Y} \overline{H_\z(\tau)}\overline{G_z(\tau)}v^{-s-1}dudv
%$$
%is holomorphic for $\re(s)>1$ since we have $G_z(\tau)=v+O(1)$ as $v\rightarrow\infty$. The analytic continuation to $s=0$ of the second summand of \label{} vanishes 
%\begin{multline*}
% \lim_{T\rightarrow\infty}\int_{\partial(\mathcal{F}_T\setminus \mathcal{F}_Y)} G_z(\tau)H_\z(\tau)v^{-s}dudv \\
%= \lim_{T\rightarrow\infty}\left(\int_0 ^1 G_z (u+iT)H_\z(u+iT)T^{-s}du-\int_0 ^1 G_z (u+iY)H_\z(u+iY)Y^{-s}du\right)
%\end{multline*}
%The first summand vanishes for $\re s >1$, since $G_z(\tau)=v+O(1)$ as $v\rightarrow\infty$. The second summand is analytic in $s=0$ and 
%All in all we have
%\begin{equation}\label{Yeq}
%\text{CT}_{s=0}\int_{\mathcal{F}\setminus \mathcal{F}_Y}H_z(\tau)\overline{H_\z(\tau)}v^{-s}d\tau = -\int_0 ^1 G_z (u+iY)H_\z(u+iY)du
%\end{equation}

To compute the first summand of \eqref{intsplit}, note that the functions $H_z$ and $H_\z$ have simple poles only at $z$ and $\z$. Thus the regularized inner product equals
\begin{multline*}
\left\langle H_z, H_\z\right\rangle= \text{CT}_{(s_1,s_2)=(0,0)}\int_\mathcal{F}H_z(\tau)|X_z(\tau)|^{2s_1}|X_\z(\tau)|^{2s_2}\overline{H_\z(\tau)}dudv\\
=-\text{CT}_{(s_1,s_2)=(0,0)}\int_{\mathcal{F}}H_z(\tau)|X_z(\tau)|^{2s_1}|X_\z(\tau)|^{2s_2}\xi_0\left(\overline{G_\z(\tau)}\right)dudv,
\end{multline*}
By Stokes's Theorem, the integral equals
\begin{equation}\label{polesplit}
-\int_{ \mathcal{F}}H_z(\tau)\overline{\xi_0 \left(|X_z(\tau)|^{2s_1}|X_\z(\tau)|^{2s_2}\right)}G_\z(\tau)dudv -\int_{\partial  \mathcal{F}_Y} H_z(\tau)|X_z(\tau)|^{2s_1}|X_\z(\tau)|^{2s_2}G_\z(\tau)d\tau.
\end{equation}
Since there are no poles on $\mathcal{F}_Y$, the analytic continuation of the second summand is given by just plugging in $(s_1,s_2)=(0,0)$. Now note that 
$$
\int_{\partial  \mathcal{F}_Y} H_z(\tau)G_\z(\tau)d\tau = -\int_0 ^1 H_z(u+iY)G_\z(u+iY)du,
$$
which cancels with \eqref{Yeq}. Therefore the contribution from the cusp $i\infty$ vanishes.\\

We are left to compute the analytic continuation of the first summand of \eqref{polesplit}. For this, we closely follow the proof of Theorem 6.1 in \cite{bkcyc}. For $\delta>0$ and $\varrho\in\H$, we let $B_\delta(\varrho)$ denote the closed disc of radius $\delta$ around $\varrho$ and split the domain of integration into $B_\delta(z)\cap\mathcal{F}$, $B_\delta(\z)\cap\mathcal{F}$, and $\mathcal{F}\setminus\left(B_\delta(z)\cup B_\delta(\z)\right)$. The integral over $\mathcal{F}\setminus\left(B_\delta(z)\cup B_\delta(\z)\right)$, away from the poles, vanishes at $(s_1,s_2)=(0,0)$. Similarly, in the integral over $B_\delta(z)$ (resp. $B_\delta(\z)$) we can plug in $s_2=0$ (resp. $s_1=0$). By construction of $\mathcal{F}$, the points $z$ and $\z$ lie in the interior of $\Gamma_z\mathcal{F}$, resp.~$\Gamma_\z\mathcal{F}$. For $\varrho\in\{z,\z\}$, we decompose
$$
B_\delta(\varrho) = \bigcup_{M\in\Gamma_\varrho}M\left(B_\delta(\varrho)\cap\mathcal{F}\right)
$$
to write
$$
-\int_{ B_\delta(\varrho)\cap\mathcal{F}}H_z(\tau)\overline{\xi_0 \left(|X_\varrho(\tau)|^{2s}\right)}G_\z(\tau)dudv = -\frac{1}{w_\varrho}\int_{ B_\delta(\varrho)}H_z(\tau)\overline{\xi_0 \left(|X_\varrho(\tau)|^{2s}\right)}G_\z(\tau)dudv.
$$
Now using
$$
\xi_0 \left(|X_\varrho(\tau)|^{2s}\right) =-4s\eta |X_\varrho(\tau)|^{2s-2}\frac{\overline{X_\varrho(\tau)}}{(\tau-\overline{\varrho})^2},
$$
 we have to compute 
\begin{equation}\label{res}
\text{Res}_{s=0}\left(\frac{4\eta}{w_\varrho}\int_{B_\delta(\varrho)}G_z(\tau)|X_\varrho(\tau)|^{2s-2}\frac{X_\varrho(\tau)}{(\overline{\tau}-\varrho)^2}H_\z(\tau)dudv\right)
\end{equation}
for $\varrho\in\{z,\z\}$. 
Plugging in the elliptic expansions \eqref{Hellxp} and \eqref{Gellxp} around $\varrho$, we obtain
\begin{multline*}
\frac{4\eta}{w_\varrho}\int_{B_\delta(\varrho)}H_z(\tau)|X_\varrho(\tau)|^{2s-2}\frac{X_\varrho(\tau)}{(\overline{\tau}-\varrho)^2}G_\z(\tau)dudv \\
=\frac{4\eta}{w_\varrho}\int_{B_\delta(\varrho)}\left(-\frac{ \delta_{z,\varrho}w_z \eta}{ \pi }X_\varrho(\tau)^{-1}+\sum_{n\geq 0}a_{z,\varrho}(n)X_\varrho(\tau)^n\right)\frac{|X_\varrho(\tau)|^{2s-2} X_\varrho(\tau)}{(\overline{\tau}-\varrho)^2(\tau-\overline{\varrho})^2}\\
\times\left(- \frac{w_\z}{2\pi }\log(\left|X_\varrho(\tau)\right|)+\sum_{n\geq 0}A^+_{\z,\varrho}(n)X_\varrho(\tau)^n  + \sum_{n>0}A_{\z,\varrho}^-(n)\overline{X_\varrho(\tau)}^{n}\right)dudv
\end{multline*}
We substitute $X_\varrho(\tau)=Re(\theta)$ and use $\frac{4\eta^2}{|\tau-\overline{\varrho}|^4}dudv = 2\pi R \,d\theta dR$ to obtain
\begin{align*}
\frac{2\pi}{\eta w_\varrho}&\int_0 ^\delta \int_0 ^1 \left(-\frac{\delta_{z,\varrho}w_z \eta}{ \pi }+\sum_{n\geq 0}a_{z,\varrho}(n)R^{n+1}e((n+1)\theta)\right)R^{2s-1}\\
&\qquad \qquad\times\left(- \frac{\delta_{\z,\varrho}w_\z}{2\pi }\log(R)+\sum_{n\geq 0}A^+_{\z,\varrho}(n)R^ne(n\theta)  + \sum_{n>0}A_{\z,\varrho}^-(n)R^ne(-n\theta)\right)d\theta dR \\
&=\int_0 ^\delta \left( \frac{\delta_{z,\varrho}\delta_{\z,\varrho} w_z w_\z}{\pi  w_\varrho}\log(R)- \frac{2\delta_{z,\varrho}w_z}{w_\varrho} A^+_{\z,\varrho}(0)+ \frac{2\pi}{\eta w_\varrho}\sum_{n\ge 0}A_{\z,\varrho}^-(n+1)a_{z,\varrho}(n)R^{2n+2}\right)R^{2s-1}dR\\
&=\frac{\delta_{z,\varrho}\delta_{\z,\varrho}w_z}{\pi }\int_0 ^\delta \log(R)R^{2s-1}dR -\delta_{z,\varrho} A^+_{\z,\varrho}(0)\frac{\delta^{2s}}{s}+\frac{2\pi}{\eta w_\varrho}\sum_{n\ge 0}A_{\z,\varrho}^-(n+1)a_{z,\varrho}(n)\frac{\delta^{2(n+s+1)}}{2(n+s+1)}.
\end{align*}
The last sum is analytic in $s=0$ and we have
$$
\int_0 ^\delta \log(R)R^{2s-1}dR = \frac{\delta^{2s}\log(\delta)}{2s}-\frac{\delta^{2s}}{4s^2}=-\frac{1}{4s^2}+O(1),
$$
so only the second term contributes to the residue in \eqref{res}, yielding the statement.
\end{proof}

\begin{proof}[Proof of Theorem \ref{rip}]
It follows from Lemmas \ref{ftrp} and \ref{hprod} that
$$
\left\langle f_{d}, f_{\delta}\right\rangle 
= \sum_{Q\in\mathscr{Q}_{d}/\SL_2(\Z)\atop \mathcal{Q}\in\mathscr{Q}_{\delta}/\SL_2(\Z)}\frac{1}{w_{z_Q}w_{z_\mathcal{Q}}}\left\langle H_{z_Q}, H_{z_\mathcal{Q}}\right\rangle = -\sum_{Q\in\mathscr{Q}_{d}/\SL_2(\Z)\atop \mathcal{Q}\in\mathscr{Q}_{\delta}/\SL_2(\Z)}\frac{1}{w_{z_Q}w_{z_\mathcal{Q}}}A_{z_\mathcal{Q},z_Q}^+(0).
$$
\begin{itemize}
	\item[(i)] Note that if two quadratic forms $Q$, $\mathcal{Q}$ have the same CM-point, then one has to be an integer multiple of the other. The factor has to be $\sqrt{\frac{\delta}{d}}$. So conversely, if $\frac{\delta}{d}$ is not a square, then we have
	$$
	A_{z_\mathcal{Q},z_Q}^+(0) = -\frac{1}{2\pi}\log\left(\left|J(z_Q)-J(z_\mathcal{Q})\right|^{\frac{1}{w_Qw_\mathcal{Q}}}\right)
	$$
	for any $Q\in\mathscr{Q}_{d}$ and $\mathcal{Q}\in\mathscr{Q}_{\delta}$ by the first case of \eqref{a0}.
	\item[(ii)] By the same argument as above, if neither $\frac{d}{3}$ nor $\frac{d}{4}$ is a square, then neither $\rho$ nor $i$ is a CM-point of any quadratic form of discriminant $d$. Thus by the first two cases of \eqref{a0}, we have for any $Q\neq \mathcal{Q}\in\mathscr{Q}_{d}$ 
	$$
	A_{z_\mathcal{Q},z_Q}^+(0) = -\frac{1}{2\pi}\log\left|J(z_Q)-J(z_\mathcal{Q})\right|
	$$
	and 
	$$
	A_{z_Q,z_Q}^+(0) = -\frac{1}{2\pi}\log\left|2\im(z_Q)J'(z_Q)\right|=-\frac{1}{2\pi}\log\left|\sqrt{|d|}\frac{J'(z_Q)}{Q(1,0)}\right|.
	$$
	\item[(iii)] This follows directly from the last two cases of $A_{z_Q,z_\mathcal{Q}}^+(0)$ given in \eqref{a0}.
\end{itemize}
\end{proof}

\end{document}